\documentclass[runningheads,a4paper]{llncs}

\usepackage[dvips]{epsfig}
\usepackage{amsfonts}
\usepackage{amssymb}
\usepackage{amscd}
\usepackage{amstext}
\usepackage{amsmath}
\usepackage{pstricks}
\usepackage{enumerate}

\newcommand{\calL}{{\mathcal L}}
\newcommand{\calM}{{\mathcal M}}
\newcommand{\calH}{{\mathcal H}}

\newcommand{\calD}{{\mathcal D}}

\newcommand{\calS}{{\mathcal S}}
\newcommand{\calU}{{\mathcal U}}

\newcommand{\N}{{\mathbb N}}

\newcommand{\Q}{{\mathbb Q}}

\newcommand{\Frac}[2]{\displaystyle \frac{#1}{#2}}
\newcommand{\Sum}[2]{\displaystyle{\sum_{#1}^{#2}}}
\newcommand{\Prod}[2]{\displaystyle{\prod_{#1}^{#2}}}

\newcommand{\into}{{\ \rightarrow\ }}

\def\pol#1{\langle #1 \rangle}
\def\series#1{\langle\!\langle #1 \rangle\!\rangle}
  \def\ncp#1#2{#1\langle #2\rangle}
\def\ncs#1#2{#1\langle\langle #2\rangle\rangle}
\def\Lyn{{\mathcal Lyn}}

\def\shuffle{\mathop{_{^{\sqcup\!\sqcup}}}} 
\def\ncp#1#2{#1\langle #2\rangle}

\gdef\stuffle{\;% 
  \setlength{\unitlength}{0.0125cm}% 
  \begin{picture}(20,10)(220,580) 
  \thinlines 
  \put(220,592){\line( 0,-1){ 10}} 
  \put(220,582){\line( 1, 0){ 20}} 
  \put(240,582){\line( 0, 1){ 10}} 
  \put(230,592){\line( 0,-1){ 10}} 
  \put(225,587){\line( 1, 0){ 10}} 
  \end{picture}\; 
}

\newcommand{\ad}{\operatorname{ad}}

\def\pol#1{\langle #1 \rangle}

\def\up#1{\raise 1ex\hbox{\footnotesize#1}}

\newcommand{\calR}{{\cal R}}

\def\Sum{\displaystyle\sum}
\def\Prod{\displaystyle\prod}
\def\Frac{\displaystyle\frac}
\def\path{\rightsquigarrow}

\def\bv{\mid}

\gdef\minishuffle{{\scriptstyle \shuffle}}  
\gdef\ministuffle{{\scriptstyle \stuffle}}

\def\scal#1#2{\langle #1\bv#2 \rangle}

\begingroup
\def\2#1{\ifnum#1<10 0\fi\the#1}
\xdef\isodayandtime{
{\2\day-\2\month-\the\year\space\2{\count0}:%
\2{\count2}}}
\endgroup
\title{Dual bases for non commutative symmetric and quasi-symmetric functions via monoidal factorization}
\titlerunning{Dual bases for non commutative symmetric and quasi-symmetric functions via monoidal factorization}
\author{G.H.E. Duchamp$^{1\,2}$, L. Kane$^{2}$, V. Hoang Ngoc Minh$^{3\,2}$, C. Tollu$^{1\,2}$}

\date{}

\authorrunning{G.H.E. Duchamp, L. Kane, V. Hoang Ngoc Minh, C. Tollu}
\institute{$^1$Universit\'e Paris XIII, 1, 93430 Villetaneuse, France.\\
$^2$LIPN - UMR 7030, CNRS, 93430 Villetaneuse, France.\\
$^3$Universit\'e Lille II, 1, Place D\'eliot, 59024 Lille, France.}

\begin{document}

\maketitle

\begin{abstract}
In this work, an effective construction, via Sch\"utzenberger's monoidal factorization, of dual  bases for the non commutative symmetric and quasi-symmetric functions is proposed.

%[\isodayandtime]

\smallskip

\end{abstract}

\section{Introduction}

Originally, ``symmetric functions'' are thought of as ``functions of the roots of some polynomial'' \cite{GKLLRT}.

The factorization formula 
\begin{eqnarray}
P(x)&=&\Prod_{\alpha\in \mathcal{O}(P)} (X-\alpha)=\sum_{j=0}^n X^{n-j}(-1)^j \Lambda_j(\mathcal{O}(P)) 
\end{eqnarray}
where $\mathcal{O}(P)$ is the (multi-)set of roots of $P$ (a polynomial) invites to consider $\Lambda_j(?)$ as a ``multiset (endo)functor''\footnote{We will not touch here on this categorical aspect.} rather than a function $K^n\into K$ ($K$ is a field where $P$ splits). But, here, $\Lambda_k(X)=0$ whenever $|X|<k$ and one would like to get the universal
formulas {\it i.e.} which hold true whatever the cardinality of $|X|$. This set of formulas is obtained as soon as the alphabet is infinite and, there, this calculus appears as an art of computing symmetric functions without using any variable\footnote{see http://mathoverflow.net/questions/123926/reference-request-lascouxs-formulas-for-chern-classes-of-tensor-products-and-sy/124172\#124172}.    
With this point of view, one sees that the algebra of symmetric functions comes equipped with many additional structures \cite{GKLLRT,DKKT,KT4,KT5,KLT} (comultiplications, $\lambda$-ring, transformations of alphabets, internal product, \ldots). For our concern here, the most important of these features is the fact that the (commutative) Hopf algebra of symmetric functions is self-dual.

At the cost of losing self-duality, features of the (Hopf) algebra of symmetric functions carry over to the noncommutative level \cite{GKLLRT}. This loss of self-duality has however a merit : allowing to separate the two sides in the factorization of the diagonal series\footnote{Sch\"utzenberger's monoidal factorization \cite{berstel_perrin,lothaire,lyndon}.}, thus giving a meaning to what could be considered a complete system of local coordinates for the Hausdorff group of the stuffle Hopf algebra. Indeed, the elements of the Hausdorff group of the (shuffle or stuffle) algebras excatly are, through the isomorphism $\ncs{A}{Y}\simeq (\ncp{A}{Y})^*$, the characters of the algebra. Then, applying $S\otimes\mathrm{Id}$ ($S\in\mathrm{Haus}(\calH)$) to the factorization
\begin{eqnarray}\label{resolution}
\sum_{w\in Y^*}w\otimes w&=&\prod_{l\in\Lyn Y}^{\searrow}\exp(s_l\otimes p_l)
\end{eqnarray}
and using the fact that $S$ is a (continuous) character, one gets a decomposition of $S$ through this complete system of local coordinates~:
\begin{eqnarray}\label{series}
S=\sum_{w\in Y^*}\scal{S}{w}\;w=\prod_{l\in\Lyn Y}^{\searrow}\exp(\scal{S}{s_l}\;p_l).
\end{eqnarray}
This fact is better understood when one considers Sweedler's dual of the  (shuffle of stuffle) Hopf algebra $\calH$, which contains as well $\mathrm{Haus}(\calH)$ and its Lie algebra, the space of infinitesimal characters. Such a character is here a series $T$ such that 
(as a linear form) 
\begin{eqnarray}
\Delta_*(T)=T\otimes \epsilon+\epsilon\otimes T
\end{eqnarray}
and one sees from this definition that such a series, as well as the characters, satisfies an identity of the type
\begin{eqnarray}
\Delta_*(S)=\sum_{i=1}^N S_i^{(1)}\otimes S_i^{(2)} 
\end{eqnarray}
for some finite double family $(S_i^{(1)},S_i^{(2)})_{1\leq i\leq N}$. Then in (\ref{series}), the character $S$ is factorized as an (infinite) product of {\it elementary} characters\footnote{which are exponentials of rank one infinitesimal characters}. This shows firstly, one can reconstruct a character from its projection onto the free Lie algebra\footnote{the map $S\mapsto\sum_{l\in\Lyn Y}\scal{S}{s_l}\;p_l$ is the projection onto the free Lie algebra parallel to the space generated by the non-primitive elements of the PBW basis} and secondly, we have at hand a resolution of unity from the process
\begin{eqnarray}
\mbox{character}\rightarrow\mbox{projection}\rightarrow\mbox{(coordinates) splitting}\rightarrow\mbox{exponentials}\rightarrow\mbox{infinite product}
\end{eqnarray}
and the key point of this resolution is exactly the system of coordinate forms provided by the dual family of any PBW homogeneous basis.  

This paper is devoted to a detailed exposition of the machinery and morphisms surrounding this resolution (Equation (\ref{resolution}))
and it is structured as follows~:
in Section \ref{background}, we give a reminder on noncommutative symmetric and quasi-symmetric functions,
in Section \ref{factorization}, we focus on the combinatorial aspects of the quasi-shuffle Hopf algebra will be introduced to obtain, via Sch\"utzenberger's monoidal factorization, a pair of bases in duality for the noncommutative symmetric and quasi-symmetric functions, encoded by words.

\section{Background}\label{background}

\subsection{Some notations and statistics about compositions}\label{notations}

For any {\it composition} $I=(i_1,\ldots,i_k)$ of strictly positive integers\footnote{{\it i.e.} $I$ is an element of the monoid $(\N_+)^*$ and the empty composition will be denoted here by $\emptyset$.}, called the {\it parts} of $I$, 
the mirror image of $I$, denoted by $\tilde I$, is the composition $ (i_k,\ldots,i_1)$. Let $I=(i_1,\ldots,i_k)\in(\N_+)^*$, the {\it length} and the  {\it weight} of $I$ are defined respectively as the numbers $l(I)=k$ and $w(I)=i_1+\ldots+ i_k$. The last part and the product of the partial sum of the entries of $I$ are defined respectively as the numbers $lp(I)=i_k$ and $\pi_u(I)=i_1(i_1+i_2)\ldots(i_1+\ldots+ i_k)$. One defines also
\begin{eqnarray}
\pi(I)=\prod_{p=1}^ki_p&\mbox{and}&sp(I)=\pi(I)l(I)!.
\end{eqnarray}
Let $J$ be a composition which is finer than $I$ and let $J=(J_1,\ldots,J_k)$ be the decomposition of $J$ such that, for any $p=1,\ldots,k,w(J_p)=i_p$. One defines
\begin{eqnarray}
l(J,I)=\Prod_{i=1}^kl(J_i),\quad lp(J,I)=\Prod_{i=1}^klp(J_i),\quad\pi_u(J,I)=\Prod_{i=1}^k\pi_u(J_i),\quad sp(J,I)=\Prod_{i=1}^ksp(J_i).
\end{eqnarray}

\subsection{Noncommutative symmetric functions}

Let $\bf k$ be a commutative $\Q$-algebra.
The algebra of noncommutative symmetric functions, denoted by
${\bf Sym_k}=({\bf k}\langle S_1,S_2,\ldots\rangle,\bullet,1)$, introduced in \cite{GKLLRT}, is the free associative algebra generated by an infinite sequence $\{S_n\}_{n\ge1}$ of non commuting inderminates also called {\it complete} homogenous symmetric functions. Let $t$ be another variable commuting with all the $\{S_n\}_{n\ge1}$. Introducing the ordinary generating series
\begin{eqnarray}\label{sigma(t)}
\sigma(t)&=&1+\sum_{n\ge1}S_nt^n,
\end{eqnarray}
other noncommutative symmetric functions can be derived by the following relations
\begin{eqnarray}\label{psi(t)}
\lambda(t)=\sigma(-t)^{-1},&\sigma(t)=\exp(\Phi(t)),
&\Frac{d}{dt}\sigma(t)=\sigma(t)\psi(t)=\psi^*(t)\sigma(t),
\end{eqnarray}
where $\Phi,\lambda,\psi$ are respectively the following ordinary generating series
\begin{eqnarray}
\Phi(t)=\Sum_{n\ge1}\Phi_n\frac{t^n}{n},
&\lambda(t)=1+\Sum_{n\ge1}\Lambda_nt^n,&\psi(t)=\Sum_{n\ge1}\Psi_nt^{n-1}.
\end{eqnarray}
The noncommutative symmetric functions  $\{\Lambda_n\}_{n\ge1}$ are called {\it elementary} functions. The elements 	$\{\Psi_n\}_{n\ge1}$ (resp. $\{\Phi_n\}_{n\ge1}$) are called {\it power sums of the first kind} (resp. {\it second kind}).

Let $I = (i_1,\ldots,i_k)\in(\N_+)^*$, one defines the products of complete and elementary symmetric functions,
and the products of power sums  as follows \cite{GKLLRT}
\begin{eqnarray}
S^I =S_{i_1}\ldots S_{i_k},\quad\Lambda^I =\Lambda_{i_1}\ldots\Lambda_{i_k},
\quad\Psi^I =\Psi_{i_1}\ldots\Psi_{i_k},\quad\Phi^I =\Phi_{i_1}\ldots\Phi_{i_k}.\label{11}
\end{eqnarray}
and it is established that
\begin{eqnarray}
S^I=\Sum_{J\succeq I}(-1)^{l(J)-w(I)}\Lambda^J
&\mbox{and}&
\Lambda^I=\Sum_{J\succeq I}(-1)^{l(J)-w(I)}S^J.\label{10}\\
S^I=\Sum_{J\succeq I}\frac{\Psi^J}{\pi_u(J,I)}
&\mbox{and}&
\Psi^I=\Sum_{J\succeq I}(-1)^{l(J)-l(I)}lp(J,I)S^J,\label{12}\\
S^I=\Sum_{J\succeq I}\frac{\Phi^J}{sp(J,I)}
&\mbox{and}&
\Phi^I=\Sum_{J\succeq I}(-1)^{l(J)-l(I)}\frac{\pi(I)}{l(J,I)}S^J,\label{13}\\
\Lambda^I=\Sum_{J\succeq I}(-1)^{w(J)-l(I)}\frac{\Psi^J}{\pi_u(\tilde J,\tilde I)}
&\mbox{and}&
\Psi^I=\Sum_{J\succeq I}(-1)^{w(I)+l(J)}lp(\tilde J,\tilde I)\Lambda^J,\\
\Lambda^I=\Sum_{J\succeq I}(-1)^{w(J)-l(I)}\frac{\Phi^J}{sp(J,I)}
&\mbox{and}&
\Phi^I=\Sum_{J\succeq I}(-1)^{w(J)-l(I)}\frac{\pi(I)}{l(J,I)}\Lambda^J,
\end{eqnarray}
The families $\{S^I\}_{I\in(\N_+)^*}$, $\{\Lambda^I\}_{I\in(\N_+)^*},\{\Psi^I\}_{I\in(\N_+)^*}$ and $\{\Phi^I\}_{I\in(\N_+)^*}$ are then homogeneous bases of ${\bf Sym_k}$.
Recall that $S^\emptyset=\Lambda^\emptyset=\Psi^\emptyset=\Phi^\emptyset=1$.

The ${\bf k}$-algebra ${\bf Sym_k}$ possesses a finite-dimensional grading by the weight function defined, for any composition $I = (i_1,\ldots,i_k)$,  by the number $w(S_I) = w(I)$. Its homogeneous component of weight $n$ (free and finite-dimensional) will be denoted by ${\bf Sym_k}_n$ and one has
\begin{eqnarray}
{\bf Sym_k}&=&{\bf k}1_{\bf Sym_k}\oplus\bigoplus_{n\ge1}{\bf Sym_k}_n.
\end{eqnarray}

One can also endow $\bf Sym_k$ with a structure of Hopf algebra, the coproduct $\Delta_{\star}$ being defined by one of the following equivalent formulae, with the convention that $S_0=S^\emptyset=1$ and $\Lambda_0=\lambda^\emptyset=1$ \cite{GKLLRT}
\begin{eqnarray}
\Delta_{\star}S_n=\sum_{i=0}^nS_i\otimes S_{n-i}&\mbox{and}&
\Delta_{\star}\Lambda_n=\sum_{i=0}^n\Lambda_i\otimes\Lambda_{n-i},\\
\Delta_{\star}\Psi_n=1\otimes\Psi_n+\Psi_n\otimes1&\mbox{and}&
\Delta_{\star}\Phi_n=1\otimes\Phi_n+\Phi_n\otimes1.
\end{eqnarray}
In other words, for the coproduct $\Delta_{\star}$, the power sums of the first kind $\{\Psi_n\}_{n\ge1}$ and of the second kind  $\{\Phi_n\}_{n\ge1}$ are primitive. The noncommutative symmetric function $S_1=\Lambda_1$ is primitive but  $\{S_n\}_{n\ge2}$ and $\{\Lambda_n\}_{n\ge2}$ are neither primitive nor group-like. Moreover, by (\ref{10}),  (\ref{12}) and (\ref{13}), one has
\begin{eqnarray}
S_1=\Lambda_1=\Phi_1=\Psi_1.
\end{eqnarray}

With the concatenation, the coproduct $\Delta_{\star}$ and the counit $\epsilon$ defined by
\begin{eqnarray}
\forall I\in(\N_+)^*,&&\epsilon(S^I)=\scal{S^I}{1},
\end{eqnarray}
one gets the bialgebra, $({\bf k}\langle S_1,S_2,\ldots\rangle,\bullet,1,\Delta_{\star},\epsilon)$, over the $\bf k$-algebra ${\bf Sym_k}$. This algebra, $\N$-graded by the weight is, as we will see in Theorem \ref{isomorphisms}, the {\it concatenation} Hopf algebra.

\subsection{Quasi-symmetric functions}

Let us consider also an infinite sequence $\{M_n\}_{n\ge1}$ of non commuting inderminates generating the free associative algebra\footnote{We here use the symbol $\equiv$ to warn the reader that the structure of free algebra is used to construct the basis of ${\bf QSym_k}$ which will be later free {\it as a commutative algebra} (with the stuffle product) and by no means as a noncommutative algebra (with the concatenation product).} 
${\bf QSym_k}\equiv(\ncp{{\bf k}}{M_1,M_2,\ldots},\bullet,1)$  and define the elements $\{M_I\}_{I\in(\N_+)^*}$ as follows
\begin{eqnarray}
M_{\emptyset}=1&\mbox{and}&\forall I=(i_1,\ldots,i_k)\in(\N_+)^*,\quad M_I =M_{i_1}\ldots M_{i_k}.
\end{eqnarray}
The elements $\{M_I\}_{I\in(\N_+)^*}$ of ${\bf QSym_k}$ are also called the {\it monomial quasi-symmetric} functions.
They are homogeneous polynomials of degree $w(I)$. This family is then an homogeneous basis of ${\bf QSym_k}$.

With  the pairing
\begin{eqnarray}
\forall I,J\in(\N_+)^*,&&\scal{S^I}{M_J}_{\mathrm{ext}}=\delta_{I,J},
\end{eqnarray}
one constructs the bialgebra dual to ${\bf Sym_k}$, $({\bf k}\langle M_1,M_2,\ldots\rangle,\star,1,\Delta_{\bullet},\varepsilon)$, over the $\bf k$-algebra ${\bf QSym_k}$. Here,
\begin{enumerate}
\item the coproduct $\Delta_{\bullet}$ is defined by
\begin{eqnarray}
\forall I\in(\N_+)^*,&&\Delta_{\bullet}(M_I)=\sum_{I_1,I_2\in(\N_+)^*,I_1.I_2=I}M_{I_1}\otimes M_{I_2},
\end{eqnarray}
\item the counit $\varepsilon$ is defined by
\begin{eqnarray}
\forall I\in(\N_+)^*,&&\varepsilon(M_I)=\scal{M_I}{1},
\end{eqnarray}
\item the product $\star$ is the commutative product associated to the coproduct $\Delta_{\star}$ and is defined, for any composition $I\in(\N_+)^*$, by
\begin{eqnarray}
M_I\star M_{\emptyset}=M_{\emptyset}\star M_I=M_I
\end{eqnarray}
and for any composition $I=(i,I')$ and $J=(j,J')\in(\N_+)^*$
\begin{eqnarray}
M_I\star M_J=M_i(M_{I'}\star M_J)+M_j(M_I\star M_{J'})+M_{i+j}(M_{I'}\star M_{J'}).
\end{eqnarray}
\end{enumerate}

Since the bialgebra ${\bf QSym_k}$ is $\N$-graded by the weight (as the dual of the $\N$-graded bialgebra ${\bf Sym_k}$)~:
\begin{eqnarray}
{\bf QSym_k}&=&{\bf k}1_{\bf QSym_k}\oplus\bigoplus_{n\ge1}{\bf QSym_k}_n
\end{eqnarray}
then it is, in fact, the {\it convolution} Hopf algebra. Indeed, one can check that, for any $K,I,J\in(\N_+)^*$,
\begin{eqnarray}
\scal{\Delta_{\star}S^K}{M_I\otimes M_J}_{\mathrm{ext}}
=\scal{S^K}{M_I\star M_J}_{\mathrm{ext}}&\mbox{and}&
\scal{\Delta_{\bullet}M_K}{S^I\otimes S^J}_{\mathrm{ext}}=\scal{M_K}{S^IS^J}_{\mathrm{ext}}.
\end{eqnarray}

\section{Noncommutative symmetric, quasi-sym\-me\-tric functions, and monoidal factorization}\label{factorization}

\subsection{Combinatorics on shuffle and stuffle Hopf algebras}\label{shuffles}

Let $Y=\{y_i\}_{i\ge1}$ be a totally ordered alphabet\footnote{by $y_1>y_2>y_3>\ldots$.}. The free monoid and the set of Lyndon words, over $Y$, are denoted respectively by $Y^*$ and $\Lyn Y$ \cite{berstel_perrin,lothaire,lyndon}.
The neutral element of $Y^*$ is denoted by $1_{Y^*}$.

Let $u=y_{i_1}\ldots y_{i_k}\in Y^*$, the {\it length} and the  {\it weight} of $u$ are defined respectively as the numbers $l(u)=k$ and $w(u)=i_1+\ldots+ i_k$.

Let us define the commutative  product over ${\bf k}Y$, denoted by $\mu$, as follows \cite{BDHMT}
\begin{eqnarray}
\forall y_n,y_m\in Y,&&\mu(y_n, y_m)=y_{n+m},
\end{eqnarray}
or by its associated coproduct, $\Delta_+$, defined by
\begin{eqnarray}
\forall y_n\in Y,&&\Delta_+y_n=\sum_{i=1}^{n-1}y_i\otimes y_{n-i}
\end{eqnarray}
satisfying,
\begin{eqnarray}
\forall x,y,z\in Y,\quad\scal{\Delta_+x}{y\otimes z}&=&\scal{x}{\mu(y,z)}.
\end{eqnarray}

Let ${\bf k}\langle Y\rangle$ be equipped by
\begin{enumerate}
\item The concatenation (or by its associated coproduct, $\Delta_{\bullet}$).
\item The {\it shuffle} product, {\it i.e.} the commutative product defined by \cite{reutenauer}
\begin{eqnarray}
\forall w\in Y^*,&&w\shuffle 1_{Y^*}=1_{Y^*}\shuffle w=w,\\
\forall x,y\in Y,\forall u,v\in Y^*,&&xu\shuffle yv=x(u\shuffle yv)+y(xu\shuffle v)
\end{eqnarray}
or by its associated coproduct, $\Delta_{\minishuffle}$,  defined, on the letters, by
\begin{eqnarray}
\forall y_k\in Y,\quad\Delta_{\minishuffle}y_k&=&y_k\otimes1+1\otimes y_k
\end{eqnarray}
and extended by morphism. It satisfies
\begin{eqnarray}
\forall u,v,w\in Y^*,\quad\scal{\Delta_{\minishuffle}w}{u\otimes v}&=&\scal{w}{u\shuffle v}.
\end{eqnarray}

\item The {\it quasi-shuffle} product, {\it i.e.} the commutative product defined  by \cite{hoffman}, for any $w\in Y^*$,
\begin{eqnarray}
\ w\stuffle 1_{Y^*}=1_{Y^*}\stuffle w=w,
\end{eqnarray}
and, for any $y_i,y_j\in Y,\forall u,v\in Y^*$,
\begin{eqnarray}
y_iu\stuffle y_jv
&=&y_j(y_iu\stuffle v)+y_i(u\stuffle y_jv)+\mu(y_ i,y_j)(u\stuffle v),\\
&=&y_j(y_iu\stuffle v)+y_i(u\stuffle y_jv)+y_{i+j}(u\stuffle v)\end{eqnarray}
or by its associated coproduct, $\Delta_{\ministuffle}$,  defined, on the letters, by
\begin{eqnarray}
\forall y_k\in Y,\quad\Delta_{\ministuffle}y_k&=&\Delta_{\minishuffle}y_k+\Delta_+y_k
\end{eqnarray}
and extended by morphism. It satisfies
\begin{eqnarray}
\forall u,v,w\in Y^*,\quad\scal{\Delta_{\ministuffle}w}{u\otimes v}&=&\scal{w}{u\stuffle v}.
\end{eqnarray}
Note that $\Delta_{\ministuffle}$ and $\Delta_{\minishuffle}$ are morphisms for the concatenation (by definition) whereas $\Delta_+$ is not a morphism for the product of ${\bf k}Y$ (for example $\Delta_+(y_1^2)=y_1\otimes y_1$, whereas $\Delta_+(y_1)^2=0$).
\end{enumerate}
Hence, with the counit $\tt e$ defined by
\begin{eqnarray}\label{counity}
\forall P\in{\bf k}\langle Y\rangle,&&{\tt e}(P)=\scal{P}{1_{Y^*}},
\end{eqnarray}
one gets two pairs of mutually dual bialgebras
\begin{eqnarray}
\calH_{\minishuffle}=({\bf k}\langle Y\rangle,\bullet,1,\Delta_{\shuffle},{\tt e})&\mbox{and}&
\calH_{\minishuffle}^{\vee}=({\bf k}\langle Y\rangle,\shuffle,1,\Delta_{\bullet},{\tt e}),\\
\calH_{\ministuffle}=({\bf k}\langle Y\rangle,\bullet,1,\Delta_{\stuffle},{\tt e})&\mbox{and}&
\calH_{\ministuffle}^{\vee}=({\bf k}\langle Y\rangle,\stuffle,1,\Delta_{\bullet},{\tt e}).
\end{eqnarray}
Let us then consider the following diagonal series\footnote{We use two notations for the same combinatorial object in order to stress the fact that the treatment will be slightly different.}
\begin{eqnarray}
\calD_{\minishuffle}\ =\ \sum_{w\in Y^*}w\otimes w&\mbox{and}&\calD_{\ministuffle}\ =\ \sum_{w\in Y^*}w\otimes w.
\end{eqnarray}
Here, in $\calD_{\minishuffle}$ and $\calD_{\ministuffle}$,  the operation on the right factor of the tensor product is the concatenation, and the operation on the left factor is the shuffle and the quasi-shuffle, respectively.

By the Cartier-Quillen-Milnor-Moore theorem (see \cite{BDHMT}), the connected $\N$-graded, co-commutative Hopf algebra $\calH_{\minishuffle}$ is isomorphic to the enveloping algebra of the Lie algebra of its  primitive elements which is equal to
${\calL ie}_{\bf k}\langle Y\rangle$~:
\begin{eqnarray}
\calH_{\minishuffle}\cong\calU({\calL ie}_{\bf k}\langle Y\rangle)
&\mbox{and}&
\calH_{\minishuffle}^{\vee}\cong\calU({\calL ie}_{\bf k}\langle Y\rangle)^{\vee}.
\end{eqnarray}
Hence, let us consider
\begin{enumerate}
\item the PBW-Lyndon basis $\{p_w\}_{w\in Y^*}$ for ${\cal U}({\calL ie}_{\bf k}\langle Y\rangle)$
constructed recursively as follows \cite{lyndon}
$$\left\{\begin{array}{llll}
p_y&=&y& \mbox{for }y\in Y,\\
p_{l}&=&[p_s,p_r]&\mbox{for }l\in\Lyn Y,\mbox{ standard factorization of }l=(s,r),\\
p_{w}&=&p_{l_1}^{i_1}\ldots p_{l_k}^{i_k}&\mbox{for }w=l_1^{i_1}\ldots l_k^{i_k},l_1>\ldots>l_k,l_1\ldots,l_k\in\Lyn Y,
\end{array}\right.$$
\item and, by duality\footnote{The dual family (i.e. the set of coordinate forms) of a basis lies in the algebraic dual which is here the space of noncommutative series, but as the enveloping algebra under consideration is graded in finite dimensions (here by the multidegree), these series are in fact (multihomogeneous) polynomials.}, the linear basis
$\{s_w\}_{w\in Y^*}$ for $({\bf k}\langle Y\rangle,\shuffle,1_{Y^*})$, {\it i.e.}
\begin{eqnarray}
\forall u,v\in Y^*,&&\scal{p_u}{s_v}=\delta_{u,v}.
\end{eqnarray}
It can be shown that this basis can be computed recursively as follows \cite{reutenauer}
$$\left\{\begin{array}{llll}
s_y&=&y,&\mbox{for }y\in Y,\\
s_l&=&ys_u,&\mbox{for }l=yu\in\Lyn Y,\\
s_w&=&\Frac{1}{i_1!\ldots i_k!}s_{l_1}^{\shuffle i_1}\shuffle\ldots\shuffle s_{l_k}^{\shuffle i_k}
&\mbox{for }w=l_1^{i_1}\ldots l_k^{i_k},l_1>\ldots>l_k.
\end{array}\right.$$
\end{enumerate}
Hence, we get Sch\"utzenberger's factorization of $\calD_{\minishuffle}$
$$\begin{array}{lllll}\label{factorization1}
\calD_{\minishuffle}&=&\Prod_{l\in\Lyn Y}^{\searrow}\exp(s_l\otimes p_l)&\in& \calH_{\minishuffle}^{\vee}\hat\otimes\calH_{\minishuffle}.
\end{array}$$

Similarly, by the Cartier-Quillen-Milnor-Moore theorem (see \cite{BDHMT}), the connected $\N$-graded, co-commutative Hopf algebra $\calH_{\ministuffle}$ is isomorphic to the enveloping algebra of  its  primitive elements~:
\begin{eqnarray}
\mathrm{Prim}(\calH_{\ministuffle})=\mathrm{Im}(\pi_1)=\mathrm{span}_{\bf k}\{\pi_1(w)\vert{w\in Y^*}\},
\end{eqnarray}
where, for any $w\in Y^*,\pi_1(w)$ is obtained as follows \cite{acta}
\begin{eqnarray}\label{pi1}
\pi_1(w)&=&w+\sum_{k\ge2}\frac{(-1)^{k-1}}k\sum_{u_1,\ldots,u_k\in
Y^+}\scal{ w}{u_1\stuffle\ldots\stuffle u_k}\;u_1\ldots u_k.
\end{eqnarray}
note that the eq. \ref{pi1} is equivalent to the following identity which will be used later on
\begin{eqnarray}\label{exp1}
w&=&\sum_{k\ge0}\frac1{k!}\sum_{u_1,\ldots,u_k\in Y^*}
\scal{ w}{u_1\stuffle\ldots\stuffle u_k}\;\pi_1(u_1)\ldots\pi_1(u_k).
\end{eqnarray}
In particular, for any $y_k\in Y$, the primitive polynomial $\pi_1(y_k)$ is given by
\begin{eqnarray}\label{pi1bis}
\pi_1(y_k)&=&y_k+\sum_{l\ge2}\frac{(-1)^{l-1}}{l}\sum_{j_1,\ldots,j_l\ge1\atop j_1+\ldots+j_l=k}y_{j_1}\ldots y_{j_l},
\end{eqnarray}
As previously, (\ref{pi1bis}) is equivalent to
\begin{eqnarray}
y_n&=&\sum_{k\ge1}\frac{1}{k!}\sum_{s'_1+\cdots +s'_k=n}\pi_1(y_{s'_1})\ldots\pi_1(y_{s'_k})\label{pi1ter}.
\end{eqnarray}
Hence, by introducing the new alphabet $\bar Y=\{\bar y\}_{y\in Y}=\{\pi_1(y)\}_{y\in Y}$, one has
\begin{eqnarray}
\calH_{\ministuffle}\cong\calU({\calL ie}_{\bf k}\pol{\bar Y})\cong\calU(\mathrm{Prim}(\calH_{\ministuffle}))
&\mbox{and}&
\calH_{\ministuffle}^{\vee}\cong\calU({\calL ie}_{\bf k}\pol{\bar Y})^{\vee}\cong\calU(\mathrm{Prim}(\calH_{\ministuffle}))^{\vee}.
\end{eqnarray}
By considering
\begin{enumerate}
\item the PBW-Lyndon basis $\{\Pi_w\}_{w\in Y^*}$ for $\calU(\mathrm{Prim}(\calH_{\ministuffle}))$ constructed recursively as follows \cite{acta}
$$\left\{\begin{array}{llll}
\Pi_y&=&\pi_1(y)& \mbox{for }y\in Y,\\
\Pi_{l}&=&[\Pi_s,\Pi_r]&\mbox{for }l\in\Lyn Y,\mbox{ standard factorization of }l=(s,r),\\
\Pi_{w}&=&\Pi_{l_1}^{i_1}\ldots\Pi_{l_k}^{i_k}&\mbox{for }w=l_1^{i_1}\ldots l_k^{i_k},l_1>\ldots>l_k,l_1\ldots,l_k\in\Lyn Y,
\end{array}\right.$$
\item and, by duality\footnote{{\it idem}.}, the linear basis $\{\Sigma_w\}_{w\in Y^*}$ for $({\bf k}\langle Y\rangle,\stuffle,1_{Y^*})$, {\it i.e.}
\begin{eqnarray}
\forall u,v\in Y^*,&&\scal{\Pi_u}{\Sigma_v}=\delta_{u,v}.
\end{eqnarray}
It can be shown that this basis can be computed recursively as follows \cite{BDM,acta}
$$\left\{\begin{array}{llll}
\Sigma_y&=&y&\mbox{for }y\in Y,\\
\Sigma_l&=&
\Sum_{{\{s'_1,\cdots,s'_i\}\subset \{s_1,\cdots,s_k\}, l_1\ge\cdots\ge l_n\in\Lyn Y}\atop (y_{s_1}\cdots y_{s_k})
\stackrel{*}{\Leftarrow} (y_{s'_1},\cdots,y_{s'_n},l_1,\cdots,l_n)}{\frac1{i!}y_{s'_1+\cdots+s'_i}\Sigma_{l_1\cdots l_n}}
&\mbox{for }l=y_{s_1}\cdots y_{s_k}\in\Lyn Y,\\
\Sigma_w&=&
\Frac{1}{i_1!\ldots i_k!}\Sigma_{l_1}^{\shuffle i_1}\stuffle\ldots\shuffle\Sigma_{l_k}^{\stuffle i_k}
&\mbox{for }w=l_1^{i_1}\ldots l_k^{i_k},l_1>\ldots>l_k,
\end{array}\right.$$
\end{enumerate}
we get the following extended Sch\"utzenberger's factorization of $\calD_{\ministuffle}$ \cite{BDM,acta}
\begin{eqnarray}\label{factorization2}
\calD_{\ministuffle}\ =\ \prod_{l\in\Lyn Y}^{\searrow}\exp(\Sigma_l\otimes\Pi_l)\ \in\ \calH_{\ministuffle}^{\vee}\hat\otimes\calH_{\ministuffle}.
\end{eqnarray}

\subsection{Encoding noncommutative symmetric and quasi-sym\-metric functions by words}\label{combinatoire}

\begin{proposition}\label{Y}
Let ${\cal Y}(t)$ be the following ordinary generating series of $\{y_n\}_{n\ge1}$~:
$$\begin{array}{rllcr}
{\cal Y}(t)&=&1+\Sum_{n\ge1}y_n\;t^n&\in&\Q\pol{Y}[\![t]\!].
\end{array}$$
Then ${\cal Y}(t)$ is group-like, for the coproduct $\Delta_{\ministuffle}$.
\end{proposition}

\begin{proof}
We have successively ((here, in order to make complete the correspondence ${\cal S}$, we put $y_0=1$)
\begin{eqnarray*}
\Delta_{\ministuffle}{\cal Y}(t)
=\sum_{n\ge0}\biggl(\sum_{r+s=n}y_s\otimes y_r\biggr)t^n=\sum_{n\ge0}\sum_{r+s=n}(y_s\;t^s)\otimes(y_r\;t^r).
\end{eqnarray*}
It follows then $\Delta_{\ministuffle}{\cal Y}(t)={\cal Y}(t)\hat\otimes{\cal Y}(t)$ meaning (with ${\tt e}({\cal Y}(t))=1$) that ${\cal Y}(t)$ is group-like.
\end{proof}

\begin{proposition}\label{logY}
Let $\mathcal{G}$  be the Lie algebra generated by $\{\scal{\log{\cal Y}}{t^n}\}_{n\ge1}$.
Then we have $\mathcal{G}=\mathrm{Prim}(\calH_{\ministuffle})$.
\end{proposition}

\begin{proof}
The power series $\log{\cal Y}\in\Q\pol{Y}[\![t]\!]$ is primitive then by expanding $\log{\cal Y}$, we get successively
\begin{eqnarray*}
\log{\cal Y}(t)
=\sum_{k\ge1}\frac{(-1)^{k-1}}k\biggl(\sum_{n\ge1}y_n\;t^n\biggr)^k
=\sum_{k\ge1}\frac{(-1)^{k-1}}k\biggl(\sum_{s_1,\ldots,s_n\ge1\atop s_1+\ldots+s_n=k}
y_{s_1}\ldots y_{s_n}\biggr)t^k.
\end{eqnarray*}
By (\ref{pi1bis}), we get, for any $n\ge1$, $\scal{\log{\cal Y}}{t^n}=\pi_1(y_n)$ and since $\{\pi_1(y_n)\}_{n\ge1}$ generates freely $\mathrm{Prim}(\calH_{\ministuffle})$ \cite{acta}, the expected result follows.
\end{proof}

In virtue of (\ref{pi1bis}) and (\ref{pi1ter}), we also have

\begin{corollary}
\begin{eqnarray*}
{\cal Y}(t)
&=&1+\Sum_{n\ge1}\biggl(\sum_{k\ge1}\frac1{k!}\sum_{s_1+\cdots+s_k=n}\pi_1(y_{s_1})\ldots\pi_1(y_{s_k})\biggr)t^n\\
{\cal Y}(t)^{-1}
&=&1+\Sum_{n\ge1}(-1)^n\biggl(\sum_{k\ge1}\frac1{k!}\sum_{s_1+\cdots+s_k=n}\pi_1(y_{s_1})\ldots\pi_1(y_{s_k})\biggr)t^n.
\end{eqnarray*}
\end{corollary}

\begin{corollary}\label{Yinverse}
Let us write the (group-like) power series ${\cal Y}^{-1}$ and its differentiation as follows 
$$\begin{array}{rllllllll}
{\cal Y}(t)^{-1}&=&1+\Sum_{n\ge1}X_n\;t^n&\in&\Q\pol{Y}[\![t]\!]
&\mbox{and}&
\dot{\cal Y}^{-1}&=&-{\cal Y}^{-1}\dot{\cal Y}{\cal Y}^{-1}.
\end{array}$$
Then, for any $n\ge1$, one has
$$\begin{array}{rlllllll}
\Sum_{i=1}^ny_iX_{n-i}&=&0&\mbox{and}&\Sum_{i=1}^nX_iy_{n-i}&=&0.
\end{array}$$
\end{corollary}

\begin{proof}
Using the identities ${\cal Y}{\cal Y}^{-1}={\cal Y}^{-1}{\cal Y}=1_{Y^*}$, the results follow immediately by identification of the coefficients of $t^n$ and by differentiation, respectively.
\end{proof}

\begin{corollary}\label{LR}
There exists two (unique and primitive) generating series $L$ and $R\in\Q\pol{Y}[\![t]\!]$ satisfying $\dot{\cal Y}=L{\cal Y}={\cal Y}R$.
Moreover, if
$$\begin{array}{rlllllll}
L(t)&=&\Sum_{n\ge1}L_n\;t^{n-1}&\mbox{and}&R(t)&=&\Sum_{n\ge1}R_n\;t^{n-1}
\end{array}$$
Then, for any $n\ge1$,
$$\begin{array}{ccccccl}
ny_n&=&\Sum_{i=0}^{n-1}L_iy_{n-1-i}&\mbox{and}&ny_n&=&\Sum_{i=0}^{n-1}y_iR_{n-1-i},\\
L_n&=&\Sum_{i=0}^{n-1}(i+1)y_{i+1}X_{n-1-i}&\mbox{and}&R_n&=&\Sum_{i=0}^{n-1}(i+1)X_{n-1-i}y_{i+1}.
\end{array}$$
\end{corollary}

\begin{proof}
On the one hand, by Proposition \ref{Y}, one has
\begin{eqnarray*}
\frac{d}{dt}{\cal Y}(t)&=&\sum_{n\ge1}ny_n\;t^{n-1}.
\end{eqnarray*}
On the other hand,  such generating series exist since
$$\begin{array}{rlllllll}
&\dot{\cal Y}&=&L{\cal Y}&\mbox{and}&\dot{\cal Y}&=&{\cal Y}R,\\
\iff&L&=&\dot{\cal Y}{\cal Y}^{-1}&\mbox{and}&R&=&{\cal Y}^{-1}\dot{\cal Y}.
\end{array}$$
Hence, identifying the coefficients of $t^n$ in these identities, the expected results follows.

Moreover, since $\Delta_{\ministuffle}$ commutes with $d/dt$ and it is a morphism for the concatenation then
$$\begin{array}{rlllllllll}
\Delta_{\ministuffle}L
&=&(\dot{\cal Y}\hat\otimes{\cal Y}+{\cal Y}\hat\otimes\dot{\cal Y})({\cal Y}^{-1}\hat\otimes{\cal Y}^{-1})
&=&\dot{\cal Y}{\cal Y}^{-1}\hat\otimes{\cal Y}{\cal Y}^{-1}+{\cal Y}{\cal Y}^{-1}\hat\otimes\dot{\cal Y}{\cal Y}^{-1}
&=&\dot{\cal Y}{\cal Y}^{-1}\hat\otimes1_{Y^*}+1_{Y^*}\hat\otimes\dot{\cal Y}{\cal Y}^{-1},\\
\Delta_{\ministuffle}R
&=&({\cal Y}^{-1}\hat\otimes{\cal Y}^{-1})(\dot{\cal Y}\hat\otimes{\cal Y}+{\cal Y}\hat\otimes\dot{\cal Y})
&=&{\cal Y}^{-1}\dot{\cal Y}\hat\otimes{\cal Y}^{-1}{\cal Y}+{\cal Y}^{-1}{\cal Y}\hat\otimes{\cal Y}^{-1}\dot{\cal Y}
&=&{\cal Y}^{-1}\dot{\cal Y}\hat\otimes1_{Y^*}+1_{Y^*}\hat\otimes{\cal Y}^{-1}\dot{\cal Y}.
\end{array}$$
Hence, $\Delta_{\ministuffle}L=1_{Y^*}\hat\otimes L+L\hat\otimes1_{Y^*}$
and $\Delta_{\ministuffle}R=1_{Y^*}\hat\otimes R+R\hat\otimes1_{Y^*}$ meaning that $L$ and $R$ are primitive.
\end{proof}

More generally, with the notations of Corollary \ref{LR}, one has 
\begin{proposition}
For any $k\ge1$, there exist two unique generating series ${\cal L}_k,{\cal R}_k\in\Q\pol{Y}[\![t]\!]$ such that
${\cal Y}^{(k)}={\cal L}_k{\cal Y}={\cal Y}{\cal R}_k$. The families $\{{\cal L}_k\}_{k\ge1}$ and $\{{\cal R}_k\}_{k\ge1}$ are defined  recursively as follows
\begin{eqnarray*}
{\cal L}_1=L&\mbox{and}&{\cal L}_k=\dot{\cal L}_{k-1}+{\cal L}_{k-1}L,\\
{\cal R}_1=R&\mbox{and}&{\cal R}_k=\dot{\cal R}_{k-1}+R{\cal R}_{k-1}.
\end{eqnarray*}
\end{proposition}

\begin{proof}
On the one hand, by Proposition \ref{Y}, one has, as in Corollary \ref{LR},
\begin{eqnarray*}
\frac{d^k}{dt^k}{\cal Y}(t)=\sum_{n\ge k}(n)_ky_n\;t^{n-k},
\end{eqnarray*}
where $(n)_k=n(n-1)\ldots(n-k)$ is the Pochhammer symbol. On the other hand, by induction
\begin{itemize}
\item For $k=1$, it is Corollary \ref{LR}.
\item Suppose the property holds for any $1\le n\le k-1$.
\item For $n=k$, such generating series exist since, by induction hypothesis,
$$\begin{array}{ccccccccl}
{\cal Y}^{(k)}&=&\dot{\cal L}_{k-1}{\cal Y}+{\cal L}_{k-1}\dot{\cal Y}
&=&\dot{\cal L}_{k-1}{\cal Y}+{\cal L}_{k-1}L{\cal Y}
&=&(\dot{\cal L}_{k-1}+{\cal L}_{k-1}L){\cal Y},\\
{\cal Y}^{(k)}&=&\dot{\cal Y}{\cal R}_{k-1}+{\cal Y}\dot{\cal R}_{k-1}
&=&{\cal Y}R{\calR}_{k-1}+{\cal Y}\dot{\cal R}_{k-1}
&=&{\cal Y}(R{\calR}_{k-1}+\dot{\cal R}_{k-1}).
\end{array}$$
Hence, ${\cal L}_k=\dot{\cal L}_{k-1}+{\cal L}_{k-1}L$ and ${\cal R}_k=R{\cal R}_{k-1}+\dot{\cal R}_{k-1}$.
\end{itemize}
\end{proof}

\begin{corollary}
For any proper power series $A,B$, let $\ad^n_{A}B$ be the iterated Lie brackets defined recursively by $\ad^0_{A}B=B$ and $\ad^{n+1}_{A}B=[\ad^n_{A},B]$, for $n\ge1$. Then, with notations of Corollary \ref{LR}, one has
\begin{eqnarray*}
{\cal L}_k=\sum_{n\ge0}\frac{\ad^n_{\log{\cal Y}}R}{n!}
&\mbox{and}&
{\cal R}_k=\sum_{n\ge0}(-1)^n\frac{\ad^n_{\log{\cal Y}}{\cal L}_k}{n!}.
\end{eqnarray*}
\end{corollary}

\begin{proof}
Since ${\cal L}_k{\cal Y}={\cal Y}{\cal R}_k$ then
${\cal L}_k={\cal Y}{\cal R}_k{\cal Y}^{-1}=\exp(\log{\cal Y}){\cal R}_k\exp(-\log{\cal Y})=\exp(\ad_{\log{\cal Y}}){\cal R}_k$ and then
${\cal R}_k={\cal Y}^{-1}{\cal L}_k{\cal Y}=\exp(-\log{\cal Y}){\cal L}_k\exp(\log{\cal Y})=\exp(\ad_{-\log{\cal Y}}){\cal L}_k$.
Expanding $\exp$, the results follow.
\end{proof}

\begin{proposition}\label{prim}
Let $\mathcal{G}$  be the Lie algebra generated by $\{R_n\}_{n\ge1}$ (resp.  $\{L_n\}_{n\ge1}$).
Then $\mathcal{G}=\mathrm{Prim}(\calH_{\ministuffle})$.
\end{proposition}

\begin{proof}
By Corollary \ref{LR}, one has on the one hand, 
\begin{eqnarray*}
\sum_{n\ge1}(\Delta_{\ministuffle}R_n)t^{n-1}
=1_{Y^*}\otimes\biggl(\sum_{n\ge1}R_n\;t^{n-1}\biggr)+\bigg(\sum_{n\ge1}R_n\;t^{n-1}\biggr)\otimes1_{Y^*}
=\sum_{n\ge1}(1_{Y^*}\otimes R_n+R_n\otimes1_{Y^*})t^{n-1}.
\end{eqnarray*}
Thus, by identifying the coefficients of $t^{n-1}$ in the first and last sums, on has
$\Delta_{\ministuffle}R_n=1_{Y^*}\otimes R_n+R_n\otimes1_{Y^*}$, meaning that $R_n$ is primitive.
On the other hand, according to basic properties of quasi-determinants (\cite{GR1,GR2}, see also \cite{GKLLRT}), one has
\begin{eqnarray*}
ny_n
=\left|\begin{matrix}
R_1&R_2&\ldots&R_{n-1}&\fbox{$R_n$}\cr
-1&R_1&\ldots&R_{n-2}&R_{n-1}\cr
0&-2&\ldots&R_{n-3}&R_{n-2}\cr
0&0&\ldots&-n+1&R_1
\end{matrix}\right|
=\left|\begin{matrix}
R_1&R_2&\ldots&R_{n-1}&\fbox{$R_n$}\cr
-1&R_1&\ldots&R_{n-2}&R_{n-1}\cr
0&-&\ldots&\frac12R_{n-3}&\frac12R_{n-2}\cr
0&0&\ldots&-1&\frac1{n+1}R_1
\end{matrix}\right|
\end{eqnarray*}
Hence, for any $J=(j_1,\ldots,j_n)\in(\N_+)^*$, by denoting $R^J=R_{j_1}\ldots R_{j_n}$, one obtains
\begin{eqnarray*}
y_n=\sum_{w(J)=n}\frac{R^J}{\pi(J)}=\frac{R_n}n+\sum_{w(J)=n,l(J)>1}\frac{R^J}{\pi(J)}.
\end{eqnarray*}
It means $y_n$ is triangular and homogeneous in weight in $\{R_k\}_{k\ge1}$. Conservely, $R_n$ is also triangular and homogeneous in weight in $\{y_k\}_{k\ge1}$. The $R_k$'s are then linearly independent and constitute a new alphabet.

In the same way, the $L_k$'s are primitive and linearly independent. The expected results follow .
\end{proof}

\begin{definition}
Let us define the families $\{\Pi^{(S)}_w\}_{w\in Y^*}$, for $S=L$ or $R$, of $\calH_{\ministuffle}$ as follows
$$\left\{\begin{array}{llllll}
\Pi^{(S)}_{y_n}&=&L_n\mbox{ if }S=L\mbox{ or }R_n\mbox{ if }S=R,&\mbox{for }y_n\in Y,\\
\Pi^{(S)}_{l}&=&[\Pi^{(S)}_s,\Pi^{(S)}_r],&\mbox{for }l\in\Lyn Y,\mbox{ with standard factorization of }l=(s,r),\\
\Pi^{(S)}_{w}&=&(\Pi^{(S)}_{l_1})^{i_1}\ldots(\Pi^{(S)}_{l_k})^{i_k},
&\mbox{for }w=l_1^{i_1}\ldots l_k^{i_k},l_1>\ldots>l_k,l_1\ldots,l_k\in\Lyn Y.
\end{array}\right.$$
\end{definition}

\begin{proposition}
Then, the families $\{\Pi^{(S)}_l\}_{l\in \Lyn Y}$ (resp. $\{\Pi^{(S)}_w\}_{w\in Y^*}$), for $S=L$ or $R$, are bases of $\mathrm{Prim}(\calH_{\ministuffle})$ (resp. $\calH_{\ministuffle}$), these bases are homogeneous in weight.
\end{proposition}

\begin{proof}
These results (homogeneity, primitivity, linear independence) can be proved by induction on the length of Lyndon words.
\end{proof}

\begin{definition}
Let $\{\Sigma^{(S)}_w\}_{w\in Y^*}$ be the family of $\calH_{\ministuffle}^{\vee}$ obtained by duality with $\{\Pi^{(S)}_w\}_{w\in Y^*}$~:
\begin{eqnarray*}
\forall u,v\in Y^*,&&\scal{\Pi^{(S)}_u}{\Sigma^{(S)}_v}=\delta_{u,v}.
\end{eqnarray*}
\end{definition}

\begin{theorem}
\begin{enumerate}
\item The family $\{\Pi^{(S)}_l\}_{l\in\Lyn Y}$ forms a basis of the Lie algebra generated by
$\mathrm{Prim}(\calH_{\ministuffle})$.
\item The family $\{\Pi^{(S)}_w\}_{w\in Y^*}$ forms a basis of ${\cal U}(\mathrm{Prim}(\calH_{\ministuffle}))$.
\item The family $\{\Sigma^{(S)}_w\}_{w\in Y^*}$ freely generates the quasi-shuffle algebra.
\item  The family $\{\Sigma^{(S)}_l\}_{l\in\Lyn Y}$ forms a transcendence basis of the quasi-shuffle algebra.
\end{enumerate}
\end{theorem}

\begin{proof}
The family $\{\Pi^{(S)}_l\}_{l\in\Lyn Y}$ of primitive upper triangular homogeneous in weight polynomials is free and
the first result follows. The second is a direct consequence of the Poincar\'e-Birkhoff-Witt theorem.
By the Cartier-Quillen-Milnor-Moore theorem, we get the third one and the last one is obtained
as a consequence of the constructions of $\{\Sigma^{(S)}_l\}_{l\in\Lyn Y}$ and $\{\Sigma^{(S)}_w\}_{w\in Y^*}$.
\end{proof}

\begin{corollary}
\begin{eqnarray*}
\calD_{\ministuffle}&=&\Prod_{l\in\Lyn Y}^{\searrow}\exp(\Sigma^{(S)}_l\otimes\Pi^{(S)}_l).
\end{eqnarray*}
\end{corollary}

Note that any word $u=y_{i_1}\ldots y_{i_k}\in Y^*$ corresponds one by one to a composition of integers $I=(i_1,\ldots,i_k)\in(\N_+)^*$ (and the empty word $1_{Y^*}$ corresponds to the empty composition $\emptyset$).
Note also that noncommutative symmetric functions and quasi-symmetric functions can be indexed by words in $Y^*$ instead of by composions in $(\N_+)^*$. Indeed, let $J$ be a composition, finer than $I$, associated to the word $v$ and let $J=(J_1,\ldots,J_k)$ be the decomposition of $J$ such that, for any $p=1,\ldots,k,w(J_p)=i_p$ and $J_p$ is associated to the word $u_p$ whose $w(u_p)=i_p$. Then $v\preceq u=u_1\ldots u_k$ is a unique factorization and this will be denoted as a bracketing of the word $v$.

\begin{example}
One has
\begin{itemize}
\item $(1,2,2)\preceq (1,(1,1),2)=(1,1,1,2)\longleftrightarrow y_1y_2y_2\preceq y_1(y_1y_1)y_2=y_1y_1y_1y_2$.
\item $(1,2,2)\preceq (1,2,(1,1))=(1,2,1,1)\longleftrightarrow y_1y_2y_2\preceq y_1y_2(y_1y_1)=y_1y_2y_1y_1$.
\item $(1,2,2)\preceq (1,(1,1),(1,1))=(1,1,1,1,1)\longleftrightarrow y_1y_2y_2\preceq y_1(y_1y_1)(y_1y_1)
\allowbreak=y_1y_1y_1y_1y_1$.
\end{itemize}
\end{example}

Hence, we can state the following

\begin{definition} 
Let $\calS$ and $\calM$ be the following linear maps
\begin{eqnarray*}
\calS:({\bf k}\langle Y\rangle,\bullet,1,\Delta_{\ministuffle},{\tt e})
&\longrightarrow&
({\bf k}\langle S_1,S_2,\ldots\rangle,\bullet,1,\Delta_{\star},\epsilon),\\
u=y_{i_1}\ldots y_{i_k}&\longmapsto&\calS(u)=S^{(i_1,\ldots,i_k)}=S_{i_1}\ldots S_{i_k},\\
\calM:({\bf k}\langle Y\rangle,\stuffle,1,\Delta_{\bullet},{\tt e})
&\longrightarrow&
({\bf k}\langle M_1,M_2,\ldots\rangle,\star,1,\Delta_{\bullet},\varepsilon),\\
u=y_{i_1}\ldots y_{i_k}&\longmapsto&\calM(u)=M_{(i_1,\ldots,i_k)}=M_{i_1}\ldots M_{i_k}.
\end{eqnarray*}
\end{definition}

\begin{theorem}\label{isomorphisms}
The maps $\calS$ and $\calM$ are isomorphisms of Hopf algebras.
\end{theorem}

\begin{corollary}
Let $\cal G$  be the Lie algebra generated by $\{\Pi_{y}\}_{y\in Y}$.
Then, we have ${\bf Sym_k}\cong{\cal U}({\cal G}).$
\end{corollary}

\begin{corollary}
The families $\{\calM(l)\}_{l\in\Lyn Y}$ and $\{\calM(\Sigma_l)\}_{l\in\Lyn Y}$ are transcendence bases of  the free commutative $\bf k$-algebra ${\bf QSym_k}$.
\end{corollary}

\begin{corollary}
Let $w=i_1\ldots i_k\in Y^*$ associated to $I=(i_1,\ldots,i_k)\in(\N_+)^*$. Then, we have
$$\begin{array}{rcccl}
S^I=\calS(w),&&\Frac{\Phi^I}{\pi(I)}=\calS(\pi_1(y_{i_1})\ldots\pi_1(y_{i_k})),&&\Psi^I=\calS(R_w).
\end{array}$$
\end{corollary}

\begin{proof}
On the one hand, the power series ${\cal Y},\log{\cal Y}$ and $L,R\in{\bf k}\pol{Y}[\![t]\!]$ are summable. On the other hand, by (\ref{sigma(t)}) and  (\ref{psi(t)}), since $\calS$ is continuous and commutes with $\log$, one can deduce
$$\begin{array}{rlrcccl}
&&\sigma(t)&=&\calS({\cal Y}(t))&=&1+\Sum_{k\ge1}\calS(y_k)t^k,\\
\Sum_{k\ge1}\frac{\Phi_k}kt^k&=&\log\sigma(t)&=&\calS(\log{\cal Y}(t))&=&\Sum_{k\ge1}\calS(\pi_1(y_k))t^k,\\
\Sum_{k\ge1}\Psi_kt^{k-1}&=&\psi(t)&=&\calS(R(t))&=&\Sum_{k\ge1}\calS(R_k)t^{k-1},\\
\Sum_{k\ge1}t^{k-1}\Psi_k^*&=&\psi^*(t)&=&\calS(L(t))&=&\Sum_{k\ge1}\calS(L_k)t^{k-1}.
\end{array}$$
Thus, the expected result follows immediately.
\end{proof}

\subsection{Dual bases for noncommutative symmetric and quasi-symmetric functions via
Sch\"utzenberger's monoidal factorization}

\begin{definition}
With the notations of (\ref{factorization2}), let us consider the following noncommutative generating series
$\{\calM(w)\}_{w\in Y^*}$ and $\{\calS(w)\}_{w\in Y^*}$
\begin{eqnarray*}
\mathrm{M}=
\Sum_{w\in Y^*}\calM(w)\;w\in{\bf QSym}_{\bf k}\series{Y}
&\mbox{and}&
\mathrm{S}=\Sum_{w\in Y^*}\calS(w)\;w\in{\bf Sym}_{\bf k}\series{Y}.
\end{eqnarray*}
\end{definition}

\begin{proposition}\label{M}
For the coproduct $\Delta_{\ministuffle}$, using (\ref{factorization2}), we obtain 
\begin{enumerate}
\item The noncommutative generating series $\mathrm{M}$ is group-like.
\item The noncommutative generating series $\log\mathrm{M}$ is primitive.
\end{enumerate}
\end{proposition}

\begin{proof}
\begin{enumerate}
\item It follows Friedrichs' criterion \cite{acta}.
\item By using the previous result and by applying the $\log$ map on the power series $\mathrm{M}$, we get the expected result.
\end{enumerate}
\end{proof}

\begin{corollary}
\begin{eqnarray*}
\mathrm{M}=\Prod_{l\in\Lyn Y}^{\searrow}\exp(\calM(\Sigma_l)\;\Pi_l)\in{\bf QSym_k}\series{Y}
&\mbox{and}&
\log\mathrm{M}=\Sum_{w\in Y^*}\calM(w)\;\pi_1(w)\in{\bf QSym_k}\langle\series{Y}.
\end{eqnarray*}
\end{corollary}

\begin{proof}
The first identity is equivalent  to the image of the diagonal series $\calD_{\ministuffle}$ by the tensor $\calM\otimes\mathrm{Id}$.

The second one is then equivalent to the image of $\log\mathrm{M}$ by the tensor $\mathrm{Id}\otimes\pi_1$.
It  is also equivalent to the image of $\calD_{\ministuffle}$ by the tensor $\calM\otimes\pi_1$.
\end{proof}

Finally, using (\ref{factorization2}) we deduce the following property which completes the formulae (120) given in \cite{GKLLRT}~:

\begin{corollary}
We have
$$\begin{array}{rrcccl}
&\Sum_{w\in Y^*}\calM(w)\;\calS(w)
&=&\Prod_{l\in\Lyn Y}^{\searrow}\exp(\calM(\Sigma_l)\;\calS(\Pi_l)
&=&\Prod_{l\in\Lyn Y}^{\searrow}\exp(\calM(\Sigma^{(S)}_l)\;\calS(\Pi^{(S)}_l),\\
\iff&\Sum_{w\in Y^*}M_w\;S_w
&=&\Prod_{l\in\Lyn Y}^{\searrow}\exp(M_{\Sigma_l}\;S_{\Pi_l})
&=&\Prod_{l\in\Lyn Y}^{\searrow}\exp(M_{\Sigma^{(S)}_l}\;S_{\Pi^{(S)}_l}).
\end{array}$$
\end{corollary}

\begin{proof}
By Theorem \ref{isomorphisms}, it is equivalent to the image of the diagonal series $\calD_{\ministuffle}$ by the tensor $\calM\otimes\calS$, or equivalently it is equivalent to the image of the power series $\mathrm{M}$ by the tensor $\mathrm{Id}\otimes\pi_1$.
\end{proof}

Note that these formulas are universal for any pair of bases in duality, compatible with monoidal factorization of $Y^*$, and they do not depend on the specific alphabets, usually denoted by $A$ and $X$, used to define $S(A)\in{\bf Sym_k}(A)$ and $M(X)\in{\bf QSym_k}(X)$.

\begin{example}[Cauchy type identity, \cite{GKLLRT}]
Let $A$ be a noncommutative alphabet and $X$ a totally ordered commutative alphabet. The symmetric functions of the noncommutative alphabet $XA$ are defined by means of% the generating series
$$\begin{array}{rllcr}
\sigma(XA;t)&=&\Sum_{n\geq 0}S_n(XA)t^n&:=&\Prod_{x\in X}^{\leftarrow}\sigma(A;xt).
\end{array}$$
Let $\{U_I\}_{I\in(\N_+)^*}$ and $\{V_I\}_{I\in(\N_+)^*}$ be two linear bases of
${\bf Sym_k}(A)$ and ${\bf QSym_k}(X)$ respectively.
The duality of these bases means that\footnote{{\it i.e.} the  formulae (120) given in \cite{GKLLRT}.}
\begin{eqnarray*}
\sigma(XA;1)=\Sum_{I\in(\N_+)^*}M_I(X)S^I(A)
=\Sum_{I\in(\N_+)^*}V_I(X)\;U_I(A).
\end{eqnarray*}

Typically, the linear basis $\{U_I\}_{I\in(\N_+)^*}$ is the basis of {\em ribbon} Schur functions $\{\mathrm{R}_I\}_{I\in(\N_+)^*}$, and, by duality, $\{V_I\}_{I\in(\N_+)^*}$ is the basis of {\em quasi-ribbon} Schur functions $\{\mathrm{F}_I\}_{I\in(\N_+)^*}$~:
\begin{eqnarray*}
\sigma(XA;1)
=\Sum_{I\in(\N_+)^*}M_I(X)\biggl[\Sum_{I,J\in(\N_+)^*\atop J\preceq I}\mathrm{R}_I(A)\biggr]
=\Sum_{J\in(\N_+)^*}\biggl[\Sum_{I,J\in(\N_+)^*\atop I\succeq J}M_I(X)\biggr]\mathrm{R}_J(A)
=\Sum_{J\in(\N_+)^*}\mathrm{F}_J(X)\mathrm{R}_J(A).
\end{eqnarray*}
\end{example}

Also, if one specializes the alphabets of the quasi-symmetric functions $\{M_I\}_{I\in(\N_+)^*}$ and $\{F_I\}_{I\in(\N_+)^*}$ to the commutative alphabet $X_q=\{1, q, q^2, \ldots\}$, then the generating series $\sigma(X_qA;t)$ can be viewed as the image of the diagonal series $\calD_{\ministuffle}$ by the tensor $f\otimes{\cal S}$, where $f:x_i\mapsto q^it$, and one has

\goodbreak

\begin{example}[Generating series of the analog Hall-Littlewood functions, \cite{GKLLRT}]
Let $X_q = 1/(1-q)$ denotes the totally ordered commutative alphabet $X_q = \{\dots < q^n < \dots < q < 1\}$.  The complete symmetric functions of the noncommutative alphabet $A/(1-q)$ are given by the following ordinary generating series
\begin{eqnarray*}
\sigma(\Frac{A}{1-q};t)=\Sum_{n\geq 0}S_n(\Frac{A}{1-q})t^n:=\Prod_{n\geq 0}^{\leftarrow}\sigma(A;q^nt).
\end{eqnarray*}
Hence, 
\begin{eqnarray*}
\sigma(\Frac{A}{1-q};1)
=\Prod_{n\geq 0}^{\leftarrow}\sum_{i\ge0}S_i\;q^{ni}
=\sum_{I=(i_1,\ldots,i_r)\in(\N_+)^*}\bigg[\sum_{n_1>\ldots>n_r\ge1}q^{n_1i_1+\ldots+n_ri_r}\biggr]S^I(A)
=\sum_{I\in(\N_+)^*}M_I(X)S^I(A),
\end{eqnarray*}
by specializing each letter $x_i\in X$ to $q^i$  in the quasi-symmetric function $M_I(X)$.
\end{example}

\section{Conclusion}
Once again, the Sch\"utzenberger's monoidal factorization plays a central role in the construction of pairs of bases in duality, as exemplified for the (mutually dual) Hopf algebras of quasi-symmetric functions (${\bf QSym_k}$) and of noncommutative symmetric functions (${\bf Sym_k}$), obtained as isomorphical images of the quasi-shuffle Hopf algebra ($\calH_{\ministuffle}$)  and its dual ($\calH_{\ministuffle}^{\vee}$), by $\calM$ and $\calS$ respectively.


\begin{thebibliography}{99}

\bibitem{berstel_perrin}{J. Berstel, D. Perrin}.--
\textit{Theory of codes}, Academic Press (1985).
%
\bibitem{BDM}{C. Bui, G. H. E. Duchamp, V. Hoang Ngoc Minh}.--
\textit{Sch\"utzenberger's factorization on the (completed) Hopf algebra of $q-$stuffle product}.
%
\bibitem{BDHMT}{C. Bui, G. H. E. Duchamp, N. Hoang, V. Hoang Ngoc Minh, C. Tollu}.--
\textit{Combinatorics of $\phi$-deformed stuffle Hopf algebras}.
%
\bibitem{lyndon}{K.T. Chen, R.H. Fox, R.C. Lyndon}.--
\textit{Free differential calculus, IV. The quotient groups of the lower central series}, Ann. of Math. , 68 (1958) pp. 81–95.
%
\bibitem{GKLLRT}{I.M. Gelfand, D. Krob, A. Lascoux, B. Leclerc, V.S. Retakh, J.Y. Thibon}.--
\textit{Noncommutative symmetric functions},  Advances in Mathematics 112 (1995), 218-348.
%
\bibitem{GR1}{I.M. Gelfand and V.S. Retakh}.--
\textit{Determinants of matrices over noncommutative rings}, Funct. Anal. Appl., 25, (1991), 91-102.
%
\bibitem{GR2}{I.M. Gelfand and V.S. Retakh}.--
\textit{A theory of noncommutative determinants and characteristic functions of graphs}, Funct. Anal. Appl., 26, (1992), 1-20; Publ. LACIM, UQAM, Montreal, 14, 1-26.
%
\bibitem{DKKT}{G. Duchamp, A. Klyachko, D. Krob, B. J.Y. Thibon}.--
\textit{Noncommutative symmetric functions III: Deformations of Cauchy and convolution algebras},
Discrete Mathematics and Theoretical Computer Science 1 (1997), 159--216.
%
\bibitem{KT4}{D. Krob, J.Y. Thibon}.--
\textit{Noncommutative symmetric functions IV: Quantum linear groups and Hecke algebras at $q=0$},
Journal of Algebraic Combinatorics 6 (1997), 339--376.
%
\bibitem{KT5}{D. Krob, J.Y. Thibon}.--
\textit{Noncommutative symmetric functions V: a degenerate version of $U_q(gl_N)$},
Internat. J. Alg. Comput. 9 (1999), 405-430.
%
\bibitem{KLT}{D. Krob, B. Leclerc, J.Y. Thibon}.--
\textit{Noncommutative symmetric functions II: Transformations of alphabets}, International Journal of Algebra and Computation 7 (1997), 181-264. 
%
\bibitem{hoffman}{Hoffman, M.}.--
\textit{Quasi-shuffle products}, J. of Alg. Cominatorics, 11 (2000), pages 49-68.
%
\bibitem{acta}{Hoang Ngoc Minh}.--
\textit{On a conjecture by Pierre Cartier about a group of associators}, to appear (2013).
%
\bibitem{lothaire}{M. Lothaire}.--
\textit{Combinatorics on words}, Addison-Wesley Publishing Co., Reading, Mass., coll. « Encyclopedia of Mathematics and its Applications » 17 (1983).
%
\bibitem{reutenauer}{Reutenauer, C.}.--
\textit{Free Lie Algebras}, London Math. Soc. Monographs, New Series-7, Oxford Science Publications, (1993).
%
\bibitem{viennot}{Viennot, G}.--
\textit{Alg\`ebres de Lie Libres et Mono\"\i des Libres},
Lecture Notes in Mathematics, vol. 691, Springer, 1978. %691: 94-112, January 01, 1978
%
\end{thebibliography}
\end{document}